\documentclass[11pt]{amsart}
\usepackage{amscd,amsmath,amsthm,amssymb}
 \usepackage[colorlinks]{hyperref}
\usepackage{amsfonts,amssymb,amscd,amsmath,enumerate,verbatim,enumitem}
\usepackage{lineno}
 \def\NZQ{\mathbb}               

 \def\FF{{\NZQ F}}

 \def\frk{\mathfrak}               

 \def\mm{{\frk m}}
 
 \def\nn{{\frk n}}
 \def\G{{\mathcal G}}

 \def\opn#1#2{\def#1{\operatorname{#2}}} 
 \opn\chara{char} \opn\length{\ell} \opn\pd{pd} \opn\rk{rk}
 \opn\projdim{proj\,dim} \opn\injdim{inj\,dim} \opn\rank{rank}
 \opn\depth{depth} \opn\grade{grade} \opn\height{height}
 \opn\embdim{emb\,dim} \opn\codim{codim}
 
 \opn\Tr{Tr} \opn\bigrank{big\,rank}
 \opn\superheight{superheight}\opn\lcm{lcm}
 \opn\trdeg{tr\,deg}
 \opn\reg{reg} \opn\lreg{lreg} \opn\ini{in} \opn\lpd{lpd}
 \opn\size{size} \opn\sdepth{sdepth}
 \opn\link{link}\opn\fdepth{fdepth}\opn\lex{lex}
 \opn\tr{tr}
 \opn\div{div} \opn\Div{Div} \opn\cl{cl} \opn\Cl{Cl}
 \opn\Spec{Spec} \opn\Supp{Supp} \opn\supp{supp} \opn\Sing{Sing}
 \opn\Ass{Ass} \opn\Min{Min}\opn\Mon{Mon}
 \opn\Ann{Ann} \opn\Rad{Rad} \opn\Soc{Soc}
 \opn\Im{Im} \opn\Ker{Ker} \opn\Coker{Coker} \opn\Am{Am}
 \opn\Hom{Hom} \opn\Tor{Tor} \opn\Ext{Ext} \opn\End{End}
 \opn\Aut{Aut} \opn\id{id}
 
 \opn\nat{nat}
 \opn\pff{pf}
 \opn\Pf{Pf} \opn\GL{GL} \opn\SL{SL} \opn\mod{mod} \opn\ord{ord}
 \opn\Gin{Gin} \opn\Hilb{Hilb}\opn\sort{sort}
 \opn\PF{PF}\opn\Ap{Ap}
 \opn\aff{aff} \opn
 \con{conv} \opn\relint{relint} \opn\st{st}
 \opn\lk{lk} \opn\cn{cn} \opn\core{core} \opn\vol{vol}  \opn\inp{inp} \opn\nilpot{nilpot}
 \opn\link{link} \opn\star{star}\opn\lex{lex}\opn\set{set}
 \opn\width{wd}
 \opn\Fr{F}
 \opn\QF{QF}
 \opn\G{G}
 \opn\type{type}\opn\res{res}
 \opn\gr{gr}
 
 \def\pot#1#2{#1[\kern-0.28ex[#2]\kern-0.28ex]}

 \opn\dirlim{\underrightarrow{\lim}}
 \opn\inivlim{\underleftarrow{\lim}}
 \let\sect=\cap
 \let\dirsum=\oplus
 \let\tensor=\otimes
 
 \let\Union=\bigcup

 \let\to=\rightarrow
 
 \def\Implies{\ifmmode\Longrightarrow \else
         \unskip${}\Longrightarrow{}$\ignorespaces\fi}
 \def\implies{\ifmmode\Rightarrow \else
         \unskip${}\Rightarrow{}$\ignorespaces\fi}
 \def\iff{\ifmmode\Longleftrightarrow \else
         \unskip${}\Longleftrightarrow{}$\ignorespaces\fi}

 \let\:=\colon
 \newtheorem{Theorem}{Theorem}[section]
 \newtheorem{Lemma}[Theorem]{Lemma}
 \newtheorem{Corollary}[Theorem]{Corollary}
 \newtheorem{Proposition}[Theorem]{Proposition}
 \newtheorem{Remark}[Theorem]{Remark}

 \let\epsilon\varepsilon
 \let\kappa=\varkappa
 \def\qed{\ifhmode\textqed\fi
       \ifmmode\ifinner\quad\qedsymbol\else\dispqed\fi\fi}
 \def\textqed{\unskip\nobreak\penalty50
        \hskip2em\hbox{}\nobreak\hfil\qedsymbol
        \parfillskip=0pt \finalhyphendemerits=0}
 \def\dispqed{\rlap{\qquad\qedsymbol}}
 
 \opn\dis{dis}
 \def\pnt{{\raise0.5mm\hbox{\large\bf.}}}
 
 \opn\Lex{Lex}

\begin{document}
\title {Koszul cycles and Golod rings}

\author {J\"urgen Herzog and Rasoul Ahangari Maleki}

\address{J\"urgen Herzog, Fachbereich Mathematik, Universit\"at Duisburg-Essen, Campus Essen, 45117
Essen, Germany} \email{juergen.herzog@uni-essen.de}

\address{School of Mathematics, Institute for Research in Fundamental
Sciences (IPM), P.O. Box: 19395-5746, Tehran, Iran}
\email{rahangari@ipm.ir, rasoulahangari@gmail.com}

\dedicatory{ }

\begin{abstract}
Let $S$ be the power series ring or the polynomial ring over a field $K$ in the variables $x_1,\ldots,x_n$, and let $R=S/I$,  where $I$ is proper ideal which we assume to be graded if $S$ is the polynomial ring. We give an explicit description of the  cycles of the Koszul complex whose homology classes generate the Koszul homology of $R=S/I$ with respect to $x_1,\ldots,x_n$.  The description is given in terms of the data of the free $S$-resolution of $R$. The result is  used to determine classes of Golod ideals, among them proper ordinary  powers and proper symbolic powers of monomial ideals. Our theory is also applied to stretched local rings.
\end{abstract}

\thanks{}

\subjclass[2010]{}


\maketitle

\section*{Introduction}

Let $S$ be the power series ring or the polynomial ring over a field
$K$ in the variables $x_1,\ldots,x_n$, and let $I\subset S$  be a
proper ideal of $S$, which we assume to be a graded ideal, if $S$ is
the polynomial ring.  Consider a finitely generated $R=S/I$-module
$M$.
 The formal power series ring $P^R_M(t)=\sum_{i\geq 0}\dim_K\Tor_i^R(K,M)t^i$ is called the Poincar\'{e} series of $M$. Since $M$ is also an $S$-module, we may as well consider the Poincar\'{e} series  $P^S_M(t)$, similarly defined.  Note that $P^S_M(t)$  is a polynomial, since $S$ is regular. By a result of  Serre \cite{JPS}, there is a coefficientwise inequality
\[
P^R_K(t)\leq \frac{(1+t)^n}{1-t(P^S_R(t)-1)}.
\]
The ring $R$  (or $I$ itself) is called Golod, if equality holds. In
this case the Poincar\'{e}  series $P^R_K(t)$ is a rational
function, which in general for the Poincar\'{e} series of the
residue field $K$ is not always the case,  see \cite{An}.
Interestingly, over Golod rings we have rationality not only for
$P^R_K(t)$ but also for Poincar\'{e} series of all finitely
generated $R$-modules.

In Section 1 of this note we give a canonical and explicit
description of the cycles of the Koszul complex whose homology
classes generate the Koszul homology of $R$ with respect to
$x_1,\ldots,x_n$, see Theorem~\ref{cycle}. The description is given
in terms of the data of the free $S$-resolution of $R$, and  allows
us to identify interesting classes of Golod ideals.

The same strategy, namely to give a nice description of Koszul cycles,  has been applied by the first author and Huneke \cite{HH} to show,  among other results,  that in the polynomial case,  the powers $I^k$ of $I$  are Golod for all $k\geq 2$, provided the characteristic of $K$ is zero. In that result,  the annoying assumption that the characteristic of $K$ should be  zero, arises from the fact that if the authors use the  result from \cite{H} which says that $\chara(K)=0$,  then the  desired Koszul cycles can be described in terms of Jacobians  derived from the maps in the graded minimal free $S$-resolution of $R$.  To avoid this drawback,  we choose a different description of the Koszul cycles which can be given for  any characteristic  of the  base field,  and only depends on the given order of the variables.

Apart from an explicit description of Koszul cycles, our approach to  prove Golodness for certain classes of ideals and rings is based on the  Golod criterion \cite[Proposition 1.3]{RA}, due to the second author of this paper. He showed that if  $I\subset J\subset S$ are  ideals with  $J^2\subset I$ and such that the natural maps $\Tor_i^S(K,S/I)\to \Tor_i^S(K,S/J)$
are zero for all $i\geq 1$, then $I$ is Golod. As an easy consequence of our description of the Koszul cycles it is shown in Section~2 that   $\Tor_i^S(K,S/I)\to \Tor_i^S(K,S/d(I))$ is the zero map for all $i\geq 1$,  where $d(I)$ is the ideal generated by the elements $d^i(f_j)$  for $i=1,\ldots,n$ and  $f_1,\ldots, f_m$  a system of generators of $I$. The operator $d^i$ is defined by
\[
d^i(f)=(f(0,\ldots,0,x_i,\ldots,x_n)-f(0,\ldots,0,x_{i+1},\ldots,x_n))/x_i \quad \text{for} f\in \nn.
\]
In combination with the above Golod criterion we then obtain that $I$ is a  Golod ideal, if $d(I)^2\subset I$. If this is the case we say that $I$ is $d$-Golod. The operators $d^i$ depend on the order of the variables. If $I$ happens to be $d$-Golod after a permutation $\sigma$ of the variables, then  we say that  $I$  is $d_\sigma$-Golod, and we  call $I$ strongly $d$-Golod, if it is $d_\sigma$-Golod for any permutation $\sigma$.

As one of the main applications of this approach we  obtains (that
for a monomial ideal $I$, all  proper ordinary  powers, saturated
powers or  symbolic powers of $I$ are Golod. The same holds true for
the integral closures $\overline{I^k}$ of the powers $I^k$ for
$k\geq 2$, see Theorem~\ref{golodlist} and
Proposition~\ref{integral}. The same results can be found in
\cite{HH} for graded ideals,  but under the additional assumption
that $\chara(K)=0$. Here we have no assumptions on the the
characteristic but we  prove these results only for monomial ideal.
However, with our methods,    a new class of non-monomial Golod
ideals in the formal power series ring  is  detected, see
Proposition~\ref{new}.

As a last application of the techniques presented in this paper we have a result of more general nature. In Theorem~\ref{str} it is shown that if $R$ is a stretched local ring or a standard graded stretched $K$-algebra, then $R$ is Golod, if one of the following conditions is satisfied:
(i) $R$ is standard graded, (ii) $R$ is not Artinian, (iii) $R$ is Artinian and the socle dimension of $R$ coincides with its embedding dimension.

The result that $R$ is Golod,  if (iii) is satisfied,  has been shown in the recent paper \cite{S}.  Our proof of this case can be deduced without any big efforts from the case that $R$ is standard graded. The latter case is accessible to our theory,  since after a suitable extension of the base field, the  defining ideal of a standard graded stretched $K$-algebra $R$ turns out to be $d_\sigma$-Golod  for a suitable permutation of the variables.

\section{A description of the Koszul cycles}

Let $K$ be a field, and let $S_r$  stand for  power series ring $K[[x_r,\ldots,x_n]]$ or the
polynomial ring $K[x_r,\ldots,x_n]$ over  $K$. For $S_1$ we simply write $S$. Since $S_r$ is naturally embedded into $S$ we may view any element in $S_r$ also as an element in $S$. We denote by
$\nn=(x_1,\ldots,x_n)$  the maximal (resp.\ the graded) maximal
ideal of $S$.

\indent Let $f\in\nn$. Then
\[
f= \sum_{(r_1,\ldots,r_n)\in \mathbb{N}_0^n}\alpha_{r_1,\ldots,r_n}x_1^{r_1}\cdots x_n^{r_n}, \]
where the coefficients $\alpha_{r_1,\ldots,r_n}$ belong to $K$.

For $f\in \nn$ and $r=1,\ldots,n$, we set
\begin{eqnarray}
\label{quasipartial}
 d^r(f)=\frac{f(0,\ldots,0,x_r,\ldots,x_n)-f(0,\ldots,0,x_{r+1},\ldots,x_n)}{x_r}.
\end{eqnarray}
Then the following rules hold:
\begin{enumerate}
\item[(i)] $f=d^1(f)x_1+\cdots+d^n(f)x_n$, and
\item[(ii)]  $d^r(f)\in S_r$ for $1,\ldots, n$.
\end{enumerate}
Note that the operators $d^r$ are uniquely determined by (i) and (ii).

\medskip
The following example demonstrates this definition: let
$S=K[[x_1,x_2,x_3,x_4]]$ and $f=x_1^2x_3+x_1x_2^3 + x_2^2x_3^3 +
x_3^2x_4$.  Then $d^1(f)=x_1x_3+x_2^3,$ $d^2(f)=x_2x_3^3,
d^3(f)=x_3x_4$ and $d^4(f)=0$.

\medskip
Of course the definition of the $d^r(f)$ depend on the order of the variables.
Like partial derivatives, the operators $d_i :S \to S_i$ are $K$-linear maps, and there is a product rule which is however less simple than that for partial derivatives. Indeed one has

\begin{Lemma}\label{prod}
Let $f,g\in\nn$ and $r$ be an integer with $1\leq r\leq n$. Then
\begin{enumerate}
\item[{\em (i)}] $d^r$ is a $K$-linear map and so $d^r(f+g)=d^r(f)+d^r(g)$;\\
\item[{\em (ii)}] $d^r(fg)=d^r(f)d^r(g)x_r+\sum_{i>r} (d^r(f)d^i(g)+d^r(g)d^i(f))x_i$.
\end{enumerate}
\end{Lemma}

\begin{proof} (i) is obvious. (ii) follows easily from (\ref{quasipartial}) and the equations
\[
(fg)(0,\ldots,0,x_s,\ldots,x_n) =\sum_{i,j=s}^nd^i(f)d^j(g)x_ix_j,
\]
for $s=r$ and $s=r+1$.
\end{proof}

Let   $\Omega_1$ be the  free $S$-module of rank $n$ with basis
$dx_1,\ldots, dx_n$. We denote by $\Omega$ the exterior algebra
$\bigwedge \Omega_1$ of $\Omega_1$.  Note that $\Omega$ is a graded
$S$-algebra with graded components $\Omega_i=\bigwedge^i \Omega_1$.
In particular,  $\Omega_0=S$ and $\Omega_i$ is a free $S$-module
with basis
\[
dx_{r_1}\wedge \ldots \wedge dx_{r_i}, \quad 1\leq r_1<r_2< \ldots <r_i\leq n.
\]
Let $\partial_i:\Omega_i\rightarrow \Omega_{i-1}$ be the $S$-linear map
given by
\[ \partial_i(dx_{r_1}\wedge \ldots \wedge dx_{r_i})=\sum_{k=1}^i (-1)^{k+1}x_{r_k}dx_{r_1}\wedge\ldots \wedge dx_{r_{k-1}}\wedge dx_{r_{k+1}}\wedge\ldots\wedge dx_i\]

To simplify our notation we shall write $dx_{r_1}\ldots dx_{r_i}$
for $dx_{r_1}\wedge \ldots \wedge dx_{r_i}$.

The complex
$\Omega$ is nothing but  the Koszul complex with respect to the
sequence $x_1,\ldots, x_n$,  and so is a minimal free $S$-resolution of
the residue field $K$ of $S$.

Now let $I$  be a proper ideal of $S$ and
$(\FF,\delta)$  a minimal free $S$-resolution
of $R=S/I$. Then for each $i=0,\ldots,n$ we have the following isomorphism
\[
\psi_i:K\otimes F_i=H_i(K\otimes \FF)\cong H_i(\Omega\otimes \FF)\cong H_i(\Omega\otimes S/I).
\]
Tracing through this isomorphism one obtains $i$-cycles in $\Omega\otimes S/I$ whose homology classes form a $K$-basis of $H_i(\Omega\otimes S/I)$.

We recall that an element $z=(z_0\ldots,z_i)\in
(\Omega\otimes \FF)_i=\bigoplus_{j=0}^i
\Omega_j\otimes F_{i-j}$ is a cycle in $(\Omega\otimes\FF)$ if and only if
\begin{eqnarray}
\label{cycle-eq}
(\id\otimes\delta_{i-j})(z_j)=(-1)^{j}(\partial_{j+1}\otimes \id)(z_{j+1})\quad \text{for all} \quad  j.
\end{eqnarray}

Now the isomorphism  $\psi_i$ can be describe as follows: let $1\otimes
f\in K\otimes F_i$ and choose a cycle $z=(z_0\ldots,z_i)\in
(\Omega\otimes \FF)_i$ such that $z_0$ maps
to $1\otimes f$ under canonical epimorphism $\Omega_0\otimes
F_i\rightarrow K\tensor F_i$. Then
\[\psi_i(1\otimes
f)=[\bar{z_i}],\]
where $\bar{z_i}$ denotes the image of $z_i$ in
$\Omega_i\otimes S/I$  and $[\bar{z_i}]$ its homology class in  $H_i(\Omega\otimes S/I)$.

In order to make this description of $\psi_i$ more explicit one has to choose suitable  cycles  $z=(z_0\ldots,z_i)\in
(\Omega\otimes \FF)_i$. There are of course many choices for such cycles. In \cite{H} the partial derivatives  were used to describe these  cycles. Here we replace the partial derivatives by our $d^r$-operators.

Let $b_i$ be the rank of $F_i$. For each $i$ we  choose a basis $e_{i1},\ldots,e_{ib_i}$,  and let
$$\delta_i(e_{ij})=\sum_{k=1}^{b_i}\alpha_{kj}^{(i)}e_{{i-1}k}$$
for all $i,j$.

The following result is the crucial technical statement of this note.

\begin{Proposition}\label{cycle-form}
Consider the element
$(z_0,\ldots,z_i)\in(\Omega\otimes\FF)_i$
with $z_0=1\otimes e_{ij}$ and
\begin{eqnarray*}
z_{i-k}= \hspace{13.5cm}\\
\sum_{j_{k}=1}^{b_{k}}\bigg (\sum_{1\leq r_{k+1}<\ldots<r_i\leq n} \sum_{j_{k+1}=1}^{b_{k+1}}\cdots \sum_{j_{i-1}=1}^{b_{i-1}} d^{r_{k+1}}(\alpha^{(k+1)}_{j_{k}j_{k+1}})
\ldots d^{r_i}(\alpha^{(i)}_{j_{i-1}j}) dx_{r_{k+1}}\ldots
dx_{r_i}\bigg)\otimes e_{k j_{k}}
\end{eqnarray*}
for all $0\leq k<i$.

Then
\[ (\id\otimes\delta_{k+1})(z_{i-k-1})=(\partial_{i-k}\otimes \id)(z_{i-k})\]
for all $k$.
\end{Proposition}

\begin{proof}
 Since
$\alpha^{(i)}_{j_{i-1}j}=\sum_{r=1}^n
d^r(\alpha^{(i)}_{j_{i-1}j})x_r$ it is obvious that  $(\id\otimes
\delta_i)(z_0)=(\partial_1\otimes \id)(z_1)$. Thus  the assertion
is true for $i-k=1$. Now let  $i-k>1$. Then
\[
(\id\otimes\delta_{k})(z_{i-k}) =\sum_{j_{k-1}=1}^{b_{k-1}}h_{j_{k-1}}\tensor e_{k-1 j_{k-1}},
\]
where
\[
h_{j_{k-1}}=\sum_{1\leq r_{k+1}<\ldots<r_i\leq n} \sum_{j_{k}=1}^{b_{k}}\ldots \sum_{j_{i-1}=1}^{b_{i-1}}\alpha^{(k)}_{j_{k-1}j_{k}} d^{r_{k+1}}(\alpha^{(k+1)}_{j_{k}j_{k+1}})
\cdots d^{r_i}(\alpha^{(i)}_{j_{i-1}j})dx_{r_{k+1}}\cdots dx_{r_i}.
\]
In order to prove the assertion it suffices to show that for all $j_{k-1}=1,\ldots,b_{k-1}$ we have
\begin{equation}\label{assertion}
h_{j_{k-1}}=\partial_{i-k+1}(g_{j_{k-1}}),
\end{equation}
where
\[
g_{j_{k-1}}=\sum_{1\leq r_{k}<\ldots<r_i\leq
n} \sum_{j_{k}=1}^{b_{k}}\ldots
\sum_{j_{i-1}=1}^{b_{i-1}}d^{r_k}(\alpha^{(k)}_{j_{k-1}j_{k}})
\cdots d^{r_i}(\alpha^{(i)}_{j_{i-1}j}) dx_{r_{k}}\cdots
dx_{r_i}.
\]
By applying rule (i), we see that
\[
h_{j_{k-1}}=\sum_{1\leq r_{k+1}<\ldots<r_i\leq n}a_{ r_{k+1},\ldots, r_i},
\]
where
\[
a_{ r_{k+1},\ldots, r_i}=\sum_{r=1}^{n} \sum_{j_{k}=1}^{b_{k}}\ldots \sum_{j_{i-1}=1}^{b_{i-1}}d^{r}(\alpha^{(k)}_{j_{k-1}j_{k}}) d^{r_{k+1}}(\alpha^{(k+1)}_{j_{k}j_{k+1}})
\cdots d^{r_i}(\alpha^{(i)}_{j_{i-1}j})x_rdx_{r_{k+1}}\cdots
dx_{r_i}.
\]
Since $\alpha^{(k)}\alpha^{(k+1)}=0$, ´Lemma~\ref{prod} implies that
\begin{eqnarray}
\label{formula}
\sum_{j_{k}=1}^{b_{k}}d^{r_{k+1}}(\alpha^{(k)}_{j_{k-1}j_{k}})d^{r_{k+1}}(\alpha^{(k+1)}_{j_{k}j_{k+1}})x_{r_{k+1}}=\hspace{6.5cm}\\ -\sum_{r_{k+1}<r\leq
n}\sum_{j_{k}=1}^{b_{k}}d^{r}(\alpha^{(k)}_{j_{k-1}j_{k}})
d^{r_{k+1}}(\alpha^{(k+1)}_{j_{k}j_{k+1}})x_r-\sum_{r_{k+1}<r\leq
n}\sum_{j_{k}=1}^{b_{k}}d^{r_{k+1}}(\alpha^{(k)}_{j_{k-1}j_{k}})
d^{r}(\alpha^{(k+1)}_{j_{k}j_{k+1}})x_r. &&\nonumber
\end{eqnarray}
By using this identity, we obtain for  $a_{ r_{k+1},\ldots, r_i}$ the expression
\begin{eqnarray*}
 a_{ r_{k+1},\ldots, r_i}=\hspace{13cm}\\
\sum_{1\leq r<r_{k+1}}\sum_{j_{k}=1}^{b_{k}}\ldots \sum_{j_{i-1}=1}^{b_{i-1}}d^{r}(\alpha^{(k)}_{j_{k-1}j_{k}}) d^{r_{k+1}}(\alpha^{(k+1)}_{j_{k}j_{k+1}})
\cdots d^{r_i}(\alpha^{(i)}_{j_{i-1}j})x_{r}dx_{r_{k+1}}\cdots
dx_{r_i} \hspace{3cm}\\
-\sum_{r_{k+1}<r\leq n}^{n}\sum_{j_{k}=1}^{b_{k}}\ldots
\sum_{j_{i-1}=1}^{b_{i-1}}d^{r_{k+1}}(\alpha^{(k)}_{j_{k-1}j_{k}})
d^{r}(\alpha^{(k+1)}_{j_{k}j_{k+1}}) d^{k_{r+2}}(\alpha^{(k+2)}_{j_{k+1}j_{k+2}})\cdots \hspace{4.5cm}\\
\cdots
d^{r_i}(\alpha^{(i)}_{j_{i-1}j})x_{r}dx_{r_{k+1}}\cdots dx_{r_i}.\hspace{2cm}
\end{eqnarray*}
Next we use the analogue to formula (\ref{formula}) corresponding to the fact that  $\alpha^{(k+1)}\alpha^{(k+2)}=0$. Then the  sum in the bottom row of the previous expression for $a_{
r_{k+1},\ldots, r_i}$ can be rewritten as
\begin{eqnarray*}
\sum_{r_{k+1}<r< r_{k+2}}\sum_{j_{k}=1}^{b_{k}}\ldots
\sum_{j_{i-1}=1}^{b_{i-1}}d^{r_{k+1}}(\alpha^{(k)}_{j_{k-1}j_{k}})
d^{r}(\alpha^{(k+1)}_{j_{k}j_{k+1}}) \cdots
d^{r_i}(\alpha^{(i)}_{j_{i-1}j})x_{r}dx_{r_{k+1}}\cdots dx_{r_i}\hspace{2cm}\\
-\sum_{r_{k+2}<r\leq n}^{n}\sum_{j_{k}=1}^{b_{k}}\ldots
\sum_{j_{i-1}=1}^{b_{i-1}}d^{r_{k+1}}(\alpha^{(k)}_{j_{k-1}j_{k}})
d^{r_{k+2}}(\alpha^{(k+1)}_{j_{k}j_{k+1}})d^{r}(\alpha^{(k+2)}_{j_{k+1}j_{k+2}})d^{r_{k+3}}(\alpha^{(k+3)}_{j_{k+2}j_{k+3}})\cdots\hspace{1cm}\\
\cdots d^{r_i}(\alpha^{(i)}_{j_{i-1}j})x_{r}dx_{r_{k+1}}\cdots
dx_{r_i}.\hspace{1cm}
\end{eqnarray*}
Proceeding this way,  we obtain that
\begin{eqnarray}
\label{hj}
h_{j_{k-1}}= \sum_{1\leq r_{k+1}<\ldots<r_i\leq n} \sum_{l=k}^{i}(-1)^{l-k}c^l_{r_{k+1},\ldots,r_i},
\end{eqnarray}
where
\begin{eqnarray*}
c^l_{r_{k+1},\ldots,r_i}=\sum_{r_l<r<r_{l+1}}\sum_{j_k=1}^{b_k}\ldots \sum_{j_{i-1}=1}^{b_i}d^{r_{k+1}}(\alpha^{(k)}_{j_{k-1},j_k})\cdots d^{r_{l}}(\alpha^{(l-1)}_{j_{l-2},j_{l-1}})d^{r}(\alpha^{(l)}_{j_{l-1},j_{l}})\cdot\hspace{3cm}\\
d^{r_{l+1}}(\alpha^{(l+1)}_{j_{l},j_{l+1}})
\cdots d^{r_{i}}(\alpha^{(i)}_{j_{i-1},j})x_r dx_{r_{k+1}}\cdots dx_i.\hspace{2cm}
\end{eqnarray*}
Here, by definition,  $r_k=1$ and $r_{i+1}=n$.

\medskip
On the other hand,
\begin{eqnarray*}
\partial_{i-k+1}(g_{j_{k-1}})
 =\partial_{i-k+1}( \sum_{1\leq r^{'}_{k}<\ldots<r^{'}_i\leq n} \sum_{j_{k}=1}^{b_{k}}\ldots
\sum_{j_{i-1}=1}^{b_{i-1}}d^{r^{'}_k}(\alpha^{(k)}_{j_{k-1}j_{k}})
\ldots d^{r^{'}_i}(\alpha^{(i)}_{j_{i-1}j}) dx_{r^{'}_{k}}\ldots
dx_{r^{'}_i})\hspace{3cm} \\
= \sum_{1\leq r^{'}_{k}<\ldots<r^{'}_i\leq n} \sum_{j_{k}=1}^{b_{k}}\ldots
\sum_{j_{i-1}=1}^{b_{i-1}}\sum_{s=1}^{s=i-k+1} (-1)^{s+1}d^{r^{'}_k}(\alpha^{(k)}_{j_{k-1}j_{k}})\ldots\hspace{8cm}\\
\ldots d^{r^{'}_i}(\alpha^{(i)}_{j_{i-1}j})x_{r^{'}_{k+1-s}} dx_{r^{'}_{k}}\ldots dx_{r^{'}_{k-s}}dx_{r^{'}_{k-s+2}}
dx_{r^{'}_i} \hspace{3cm}\\
=\sum_{1\leq r^{'}_{k}<\ldots<r^{'}_i\leq n}\sum_{s=1}^{i-k+1}(-1)^{s+1}c^{{k+1-s}}_{r'_k,\ldots, r'_{k-s},r'_{k-s+2},\ldots,r'_i}.\hspace{10cm}
\end{eqnarray*}
Comparing this with (\ref{hj}) the desired equality (\ref{assertion}) follows.
\end{proof}

As a consequence of Proposition~\ref{cycle-form} we obtain

\begin{Theorem}
\label{cycle}
For $i=1,\ldots,n$ and $j=1,\ldots,b_i$ let
\[
z_{ij}= \sum_{1\leq r_{1}<\ldots<r_i\leq n} \sum_{j_1=1}^{b_1}\ldots \sum_{j_{i-1}=1}^{b_{i-1}} d^{r_{1}}(\alpha^{(1)}_{j_{0}j_{1}})
\ldots d^{r_i}(\alpha^{(i)}_{j_{i-1}j}) dx_{r_{1}}\ldots
dx_{r_i},
\]
and denote its image in $\Omega_i\tensor S/I$ by $\bar{z}_{ij}$. Then for all $i$ and $j$, the element $\bar{z}_{ij}$ are cycles of $\Omega\tensor S/I$, and  the homology classes $[\bar{z}_{ij}]$ with $j=1,\ldots,b_i$ form a $K$-basis of $H_i(\Omega\otimes S/I)$.
\end{Theorem}

\begin{proof}
The elements $1\otimes e_{ij}$ with  $1\leq j\leq b_i $ form a $K$-basis
of $K\otimes F_i$. Proposition~\ref{cycle-form} together with  (\ref{cycle-eq}) implies
that $\psi_i(1\otimes e_{ij})=\pm[\bar{z}_{ij}]$,  and this yields the assertion.
\end{proof}

\section{A Golod criterion}

Let $S$ be  as before and $R=S/I$ with $I\subset \nn$. As before we
assume that $I$ is a graded ideal,  if $S$ is the polynomial ring.

\medskip

We will use the following Golod criterion given in \cite{RA} by the
second author.

\begin{Theorem}
\label{rasoul}
Let $I\subset J\subset S$ be ideals with  $J^2\subset I$ and such that the natural maps
\[
\Tor_i^S(K,S/I)\to \Tor_i^S(K,S/J)
\]
are zero for all $i\geq 1$. Then $I$ is Golod.
\end{Theorem}

Let $I=(f_1,\ldots,f_m)$. We define  $d(I)$ to be the ideal generated by the elements $d^i(f_j)$ with  $i=1,\ldots,n$ and $j=1,\ldots,m$.
Note that  $I\subset d(I)$ and that $d(I)$ does not depend on the particular choice of the generators, but of course on the order of the variables.  If $x_{\sigma(1)},\ldots,x_{\sigma(n)}$ is a relabeling  the variables given by the permutation  $\sigma$, then for $i=1,\ldots,n$ we set
\[
d_\sigma^i(f)=\sigma(d^i(f(x_{\sigma^{-1}(1)},\ldots,x_{\sigma^{-1}(n)})),
\]
and let $d_\sigma(I)$ be ideal generated by the elements $d_\sigma^i(f_j)$.

\medskip
The important conclusion that arises from our description of the Koszul cycles is the following

\begin{Corollary}
\label{key}
Let $I$ be a proper ideal of $S$ and  $\sigma$
be any  permutation of the integers $1,\ldots,n$. Then the natural map
 \[
 \Tor^S_i(K,S/I)\rightarrow \Tor^S_i(K,S/d_{\sigma}(I))
 \]
induced by the surjection $S/I\rightarrow S/d_{\sigma}(I)$ is zero for
all $i\geq 1$.
\end{Corollary}

\begin{proof} We first notice that for any $S$-module $M$ the Koszul homology $H_i(\Omega;M)$ is functorially isomorphic to $\Tor^S_i(K,M)$. Thus it suffice to show that the map $H_i(\Omega;S/I)\to H_i(\Omega;S/d_\sigma(I))$ is the zero map for all $i\geq 1$. It is enough to show this for the basis $[z_{ij}]$ elements  of $H_i(\Omega;S/I)$ as given in Theorem~\ref{cycle}. These cycles have coefficients in $d_\sigma(I)$ and hence  their image, already in $\Omega\tensor S/d_\sigma(I)$, is zero.
\end{proof}

Now combining Theorem~\ref{rasoul} with Corollary~\ref{key} we obtain

\begin{Theorem}
\label{criterion}
Let $I\subset S$ be a proper ideal such that $d_\sigma(I)^2\subset I$ for some permutation $\sigma$. Then $I$ is  a Golod ideal.
\end{Theorem}

We say that $I$ is {\em $d_\sigma$-Golod}, if $I$ satisfies the condition of the theorem, and simply say that $I$ {\em $d$-Golod}, if $I$ is $d_\sigma$-Golod for $\sigma=\id$. Finally we say that $I$ is {\em strongly $d$-Golod}, if $I$ is $d_\sigma$-Golod for all permutations $\sigma$ of the set $[n]=\{1,2,\ldots,n\}$.

\medskip
For monomial ideals $d(I)$ can be easily computed. If $u$ is a
monomial and $r$ is the smallest integer such that $x_r$ divides
$u$, then $d^r(u)=u/x_r$ and $d^i(u)=0$ for all $i\neq r$. Thus, for
example, if $I=(x_1x_2, x_2^2)$, then $d(I)=(x_2)$, and
$d_\sigma(I)=(x_1,x_2)$ for $\sigma$ the permutation $\sigma$ with
$\sigma(1)=2$ and $\sigma(2)=1$.

\medskip
Obviously, one has the following implications:
\[
\text{$I$ is strongly $d$-Golod \implies $I$ is $d_\sigma$-Golod \implies $I$ is Golod.}
\]
In the above example, $I$ is $d$-Golod, but not $d_\sigma$-Golod. In particular, $d$-Golod does not imply strongly $d$-Golod. Also, a  Golod ideal is  in general not  $d_\sigma$-Golod for any $\sigma$.  For example, the ideal $I=(x_1x_2)$ is Golod, but not $d_\sigma$-Golod.

On the other hand, if $\chara(K)=0$ and $I$ is a monomial ideal, then $I$ is strongly Golod in  the sense of \cite{HH} if and only if it is  strongly $d$-Golod. This follows from  the remarks at the begin of Section~3 in \cite{HH}, where it is observed that  a monomial ideal $I$ is strongly Golod if and only if for all  monomial generators $u,v\in I$ and all integers $i$  and $j$ with $x_i|u$ and $x_j|v$ it follows that $uv/x_ix_j\in I$. Indeed, it is obvious that strongly Golod implies   strongly $d$-Golod. Conversely, let $u,v\in I$ be  monomials with
$x_i\mid u$ and $x_j\mid v$.  If $x_j\mid u$ or $x_i\mid v$, then  clearly  $u v/x_ix_j\in I$. Suppose now that $x_j\nmid u$ and $x_i\nmid v$.  Then   $i\neq j$,  and  we may assume that $i<j$. Choose any  permutation $\sigma$ of $[n]$  such that $\sigma(1)=i$ and $\sigma(2)=j$. Then   $d_{\sigma}(u)=u/x_i$ and $d_{\sigma}(v)=v/x_j$, and since $I$ is $d_{\sigma}$-Golod we get that  $u v/x_ix_j\in I$.

\section{Applications}

Let $I$ and $J$ be ideals of $S$. The ideal $\Union_{t\geq 1}I\: J^t$ is called the {\em saturation} of $I$ with respect to $J$. For $J=\nn$, this saturation is denoted by $\widetilde{I}$. The $k$th symbolic power of $I$, denoted by $I^{(k)}$, is the saturation of $I^k$ with respect to the ideal which is the intersection of all associated, non-minimal prime ideals of $I^k$.

As a first application we prove a  result  analogue to \cite[Theorem 2.3]{HH}, whose proof follows very much the line of arguments given there.  The new and important fact is that no assumptions on the characteristic of the base field are  made.

\begin{Theorem}\label{golodlist}
Let $I,J\subset S$ be ideals. Assume that $\sigma$ is a permutation
of the integers $[n]$. Then the following  holds:
\begin{enumerate}
\item[{\em (a)}]  If $I$ and $J$ are $d_{\sigma}$-Golod, then $I\cap J$ and $IJ$ are $d_{\sigma}$-Golod.
\item[{\em (b)}] If $I$ and $J$ are $d_{\sigma}$-Golod and $d_{\sigma}(I)d_{\sigma}(J)\subset I+J$, then $I+J$ is
$d_{\sigma}$-Golod.
\item[{\em (c)}]  If $I$ is a strongly $d$-Golod monomial ideal and $J$ is an
arbitrary monomial ideal such that $I:J=I:J^2$, then $I:J$ is
strongly $d$-Golod.
\item[{\em (d)}]  If $I$ is a monomial ideal, then $I^k$, $I^{(k)}$ and
$\widetilde{I^k}$ are strongly $d$-Golod for all $k\geq 2$.
\item[{\em (e)}]  If $I\subset J$ are monomial ideals and $I$ is
$d_{\sigma}$-Golod, then $IJ$ is $d_{\sigma}$-Golod.
\end{enumerate}
\end{Theorem}

\begin{proof}
For the proofs of (a) and (b) we may assume that $\sigma=\id$. The proof for a general permutation $\sigma$ is the same.

(a) By assumption, $d(I)^2\subset I$ and $d(J)^2\subset J$. Hence,  since $d(I\sect J)\subset d(I)\sect d(J)$, it follows that
$d(I\sect J)^2\subset d(I)^2\sect d(J)^2\subset I\sect J$, and this shows that $I\sect J$ is $d$-Golod.

Now let $f\in I$ and $g\in J$. Then Lemma~\ref{prod}(ii) implies that $d^r(fg)\in d(I)d(J)$. Therefore,  $d(IJ)\subset d(I)d(J)$, and hence, $d(IJ)^2\subset d(I)^2d(J)^2\subset IJ$, as desired.

(b) Since by assumption, $d(I)^2\subset I$, $d(J)^2\subset J$ and $d(I)d(J)\subset I+J$, it follows that $d(I+J)^2=(d(I)+d(J))^2=d(I)^2 +d(I)d(J)+d(J)^2\subset I+J$, which shows that $I+J$ is $d$-Golod.

(c) Let $w_1,w_2\in I:J$ be two monomials and $i$ and $j$ be integers with $x_i|w_1$ and $x_j|w_2$. We must show that  $w_1w_2/x_ix_j\in I:J$.

 Assume that  $u, v\in J$  are arbitrary. Therefore $w_1u\in I$ and
$w_2v\in I$.  Since $I$ is strongly $d$-Golod, it follows that
$(w_1w_2/x_ix_j)uv= (w_1u/x_i)(w_2v/x_j)\in I$. Hence,
$w_1w_2/x_ix_j\in I:J^2=I:J$.

(d) We first show that $I^k$ is strongly $d$-Golod for all $k\geq
2$. For this we have to show that $I^k$ is  $d_\sigma$-Golod for all
$k\geq 2$ and all permutations $\sigma$ of $[n]$. We prove this for
$\sigma=\id$. The proof for general $\sigma$ is the same. So now let
$w=w_1\cdots w_k$ be a monomial generator of $I^k$ with $w_j\in I$
for $j=1,\ldots k$, and let $r$ be the smallest integer which
divides $w$. We may assume that $x_r$ divides $w_1$. Then $d^i(w)=0$
for $i\neq r$ and $d^r(w)\in I^{k-1}$. It follows that
$d(I^k)\subset I^{k-1}$, and hence $d(I^k)^2\subset
I^{2(k-1)}\subset I^k$.

The remaining statements of (d) now result  from the following more general fact: let $I$ be a  strongly $d$-Golod ideal and $J$ an arbitrary monomial ideal. Then the saturation of $I$ with respect to $J$ is strongly $d$-Golod.

For the proof of this we observe that, due to the fact that $S$ is Noetherian, there exists an integer $t_0$ such that $\Union_{t\geq 0}I:J^t=I:J^s$ for $s\geq t_0$. Thus if $L=J^{t_0}$. Then $\Union_{t\geq 0}I:J^t=I:L=I:L^2$, and the claim follows from (c).

(e)  Let $u_1, u_2\in I$ and $v_1,v_2\in J$ be monomials. Assume that $i$ is the
smallest integer such that $x_{\sigma(i)}\mid u_1v_1$ and $j$ is the
smallest integer such that $x_{\sigma(j)}\mid u_2v_2$. We need to
show that $u_1v_1u_2v_2/x_{\sigma(i)}x_{\sigma(j)}\in IJ.$

If $x_{\sigma(i)}|u_1$ and $x_{\sigma(j)}|u_2$, then $(u_1u_2/x_{\sigma(i)}x_{\sigma(j)})v_1v_2\in IJ$ since $I$ is
$d_{\sigma}$-Golod.  If
$x_{\sigma(i)}|v_1$ and $x_{\sigma(j)}|v_2$, then $u_1u_2(v_1v_2/x_{\sigma(i)}x_{\sigma(j)})\in
IJ$ since $I\subset J$. If $x_{\sigma(i)}|u_1$  and $x_\sigma(j)|v_2$, then $u_1/x_{\sigma(i)}(v_1u_2v_2/x_{\sigma(j)})\in IJ$ since  $I$ is
$d_{\sigma}$-Golod. The case $x_{\sigma(i)}|v_1$  and $x_\sigma(j)|u_2$ is similar.
\end{proof}

Let $I$ be a monomial ideal. We denote by $\bar{I}$ the integral closure of $I$. The following result  and its proof are completely analogue to that of \cite[Proposition 3.1]{HH}, where a similar result  is shown for strongly Golod monomial ideals.

\begin{Proposition}
\label{integral}
Let $I$ be a monomial ideal which is $d_{\sigma}$-Golod.  Then $\bar{I}$ is $d_{\sigma}$-Golod. In
particular,  $\bar{I^k}$ is strongly $d$-Golod for all $k\geq 2$.
\end{Proposition}

A monomial ideal $I$ of polynomial ring $S=K[x_1,\ldots,x_n]$ is
called stable if for all monomial $u\in I$ one has $x_i u/x_{m(u)}$
for all $i\leq m(u)$, where $m(u)$ is the largest integer $j$ such
that $x_j$ divides $u$.

\medskip
We  use Theorem~\ref{golodlist} to reprove and generalize a result of Aramova and  \cite{AH} who showed that stable monomial ideals are
Golod.

\begin{Corollary}
Let $I$ be a stable monomial ideal. Then $IJ$ is a Golod ideal
for any $J$ with  $I\subset J$. In particular, $I$ is Golod.
\end{Corollary}

\begin{proof}
Let  $\sigma$ be  the permutation reversing the order of the variables. Then
$d_{\sigma}(u)=u/x_{m(u)}$. Hence for all monomials $u,v\in I$ one has  $d_{\sigma}(u)d_{\sigma}(v)\in I$, since $I$ is stable. It follows that $I$ is $d_\sigma$-Golod. Therefore the desired result follows from
 Theorem~\ref{golodlist}(e).
\end{proof}

In the following proposition we present new family of $d$-Golod ideals which are not necessarily monomial ideals.

\begin{Proposition}
\label{new} Let $J_i\subset K[[x_i,\ldots, x_n]]$  be an ideal for
$i=1,\ldots,n$, and $J\subset S=K[|x_1,\ldots, x_n|]$ be the ideal
generated by $\sum_{i=1}^n x_iJ_i$. Then $J^k$ is $d$-Golod for all
$k\geq 2$.
\end{Proposition}

\begin{proof}
Note that for any $k\geq 2$ we have
\[J^k=\sum_{a_1+\cdots+a_n=k}x_1^{a_1}\cdots
x_n^{a_n}J_1^{a_1}\cdots J_n^{a_n}\] where the $a_i$ are non-negative
integers.  By convention,  $x_i^{a_i}J_i^{a_i}=S$ if  $a_i= 0$. Thus we see that
 $J^k$ is generated by elements of the form
 \[
 x_1^{a_1}\cdots
x_n^{a_n} f_1\cdots f_n
\]
with integers $a_i\geq 0$ such  that $a_1+\ldots+a_n=k$ and with  $f_i\in J_i^{a_i}$ for $i=1,\ldots,n$,

One has
\[
d^i(f)=\begin{cases}
      0,  & \text {if  there exists $j<i$ with $a_j>0$ or $a_i=0$}, \\
           x_i^{a_i-1}\cdots x_n^{a_n}
f_i\cdots f_n, & \text{otherwise}.
   \end{cases}.
\]
Hence,  we see that  $d^i(f)d^j(g)\in J^k$ for all $f,g\in J^k$ and all
$i,j$ with $1\leq i,j\leq n$,  and  this shows that $J^k$ is $d$-Golod for all $k\geq
2$.
\end{proof}

\medskip
The last application we have in mind is of more general nature and deals with stretched local rings.  Let $(R,\mm,K)$ be a
Noetherian local ring with maximal ideal $\mm$
and residue field $K$ or a standard graded $K$-algebra with graded
maximal ideal $\mm$. The ring $R$ is said to be {\em stretched} if $\mm^2$ is a
principal ideal.

We set  $n=\dim_K\mm/\mm^2$ and $\tau= \dim_K\Soc(R)$,  where $\Soc(R)=(0:_R\mm)$ is the socle of $R$. Moreover, if $R$ is Artinian, we let $s$ be the largest integer such that $\mm^s\neq 0$. Note the $s+1$ is the Loewy length of $R$.

Stretched local rings have been introduced by Sally \cite{JS}. She showed that the Poincar\'{e} series of
$K$ is a rational function. Indeed,  she showed that
\begin{eqnarray}
\label{sally}
P^R_K(t)=\begin{cases}
      1/(1-nt),  & \text {if $\tau=n$} \\
            1/(1-nt+t^2), & \text{if  $\tau\neq n$}
   \end{cases}
\end{eqnarray}

Very recently in  \cite{S} it was shown  that all finitely generated
modules over a stretched Artinian local ring $R$ have a rational
Poincar\'{e} series with a common denominator by studying the  algebra  structure of the  Koszul
homology of $R$. Among other results they proved in \cite[Theorem~5.4]{S} that $R$ is Golod, if $\tau=h$. By using our methods we give an alternative proof of the result and show

\begin{Theorem}\label{str}
Let $(R,\mm,K)$ be a stretched  local ring or a stretched  standard
graded $K$-algebra. Then $R$ is Golod if one of the following conditions is  satisfied:

\medskip
\noindent
{\em (i)} $R$ is standard graded,   {\em (ii)} $R$ is not Artinian, or {\em (iii)} $R$ is Artinian and  $\tau= n$.
\end{Theorem}

The following lemma will be needed for the proof of the theorem.

\begin{Lemma}
\label{needed}
Let $R$ be a stretched  standard graded $K$-algebra with  $n\geq 2$ and
$s\geq 3$ in Artinian case. Then the following holds:
\begin{enumerate}
\item[{\em (i)}] $\tau=n$, if  $R$ is Artinian.
\item[{\em (ii)}]  $\tau=n-1$, if  $R$ is not Artinian.
\end{enumerate}
\end{Lemma}

\begin{proof}
Since $R$ is
standard graded, $R_{1}R_{i}=R_{i+1}\neq 0$ for $2\leq i< s$ if $R$ is
Artinian,  and $R_{1}R_{i}=R_{i+1}\neq 0$ for all $i\geq 2$ if  $R$ is
not Artinian. Since $\dim_K R_i=1$ for all $i\geq 2$ with $R_i\neq 0$, it follows  that $R_s\varsubsetneq \Soc(R)\varsubsetneq R_1\dirsum R_s$ if  $R$ is Artinian, and $\Soc(R)\varsubsetneq R_1$ if $R$ is not
Artinian. Thus in order to prove (i) and (ii)  we must show that $\dim \Soc(R)\sect R_1=h-1$.

After an extension of the base field we may assume that $K$ is algebraically closed. Indeed, a base field extnesion does not change the Hilbert function, not does it change the socle dimension of the $R$.

We proceed by induction on $n$. We first assume that $n=2$.
Since $\dim_K
R_{1}> \dim_K R_{2}$  and since $K$ is algebraically closed,  \cite[Lemma 2.8]{CRV} implies that  there exists a
non-zero linear form $x_1$ such that $x_1^{2}=0$. Assume that $x_1R_1\neq 0$. Then $R_2=x_1R_1$, and $R_3=R_1R_2=x_1R_1^2=x_1^2R_1=0$, contradicting the assumption that $s\geq 3$. Thus $x_1R_1=0$, and hence $x_1\in \Soc(R)$, and the assertion follows for $n=2$. Now let $n>2$, and let   $R_1'$ be a $K$-linear subspace of $R_1$ such that $R_1'\dirsum Kx_1=R_1$. Then it follows that $(R_1')^2=R_1^2$. Therefore the  standard graded $K$-algebra $R'=K[R_1']$ is a stretched $K$-algebra of embedding dimension $n-1$. By induction hypothesis, $\dim_K \Soc(R')\sect R_1'=n-2$. Let $x_2,\ldots,x_{n-1}$ span the $K$-vector space $\Soc(R')\sect R_1'$. These elements $x_i$  are also socle elements of $R$ and together with $x_1$, they span a vector space of dimension $n-1$, as desired.
\end{proof}

\begin{Remark}
\label{relations}
Suppose that $R=S/I$ is  a stretched  standard graded $K$-algebra,  where $S=K[x_1,\ldots,x_n]$ is the polynomial ring with $K$ is an algebraically closed field,  and where $I\subset \nn^2$. The proof of Lemma~\ref{needed} shows that after a suitable linear change of   coordinates one has that
\begin{enumerate}
\item[{\em (i)}] $I=(x_1,\ldots,x_{n-1})^2+x_n(x_1,\ldots,x_{n-1})+(x_n^{s+1})$, if $R$ is Artinian;

\item[{\em (ii)}] $I=(x_1,\ldots,x_{n-1})^2+x_n(x_1,\ldots,x_{n-1})$, if $R$ is not Artinian.
\end{enumerate}
\end{Remark}

\begin{proof}[Proof of Theorem~\ref{str}]
Let us first assume that $R$ is a standard graded $K$-algebra. After a suitable base field extension, which does not affect the Golod property,  we may assume that $I$ is generated as described in Remark~\ref{relations}.  We order the
variables as follows: $x_n,x_1,\ldots,x_{n-1}$, and let $\sigma$ be the corresponding permutation of $[n]$. Then  $d_{\sigma}(I)=(x_1,\ldots,x_{n-1},x_n^{s})$ in the Artinian case,  and
and $d_{\sigma}(I)=(x_1,\ldots,x_{n-1})$ in the non-Artinian case. Clearly
$d_{\sigma}(I)^2\subset I$ in both cases.  Thus in
any case $R$ is a Golod ring. This proves case (i).

In order to prove Golodness of $R$ in the case (ii), we consider the
associated graded ring $G(R)$ of $R$, which, as can be seen from its
Hilbert function, is a standard graded stretched $K$-algebra. We
claim that $G(R)$ is a Koszul algebra.  Koszulness  of a standard
graded $K$-algebra is characterized by the property that
$P^R_K(t)=1/H_R(-t)$,  where $H_R(t)$ denotes the Hilbert series for
$R$. Therefore, it is enough to prove the claim  for a suitable base
extension, because a base extension does not change the Poincar\'{e}
series of a $K$-algebra,  nor does it change its Hilbert series, as
already noticed before. Hence after this base field extension we may
assume that $I$ is generated by quadratic monomials as described  in
Remark~\ref{relations}. By a result of Fr\"oberg \cite{Fr1}, this
implies that $G(R)$ is Koszul. Now by another result of Fr\"oberg
\cite{Fr} it follows that $P^R_K(t)=P^{G(R)}_K(t)$. By case (i),
$G(R)$ is Golod, and hence
\[
P^R_K(t)=P^{G(R)}_K(t)= \frac{(1+t)^n}{1-t(P^S_{G(R)}(t)-1)}\geq
\frac{P^S_K(t)}{1-t(P^S_R(t)-1)}.
\]
The coefficientwise  inequality in these formulas follows from the
well-known fact that there is the  coefficientwise inequality
$P^S_R(t)\le P^S_{G}(R)(t)$.  Since the opposite inequality
$P^S_K(t)/(1-t(P^S_R(t)-1)\geq P^R_K(t)$ always holds, we have
equality and $R$ is Golod.

Finally suppose that (iii) is satisfied. Then $G(R)$ is a   stretched Artinian $K$-algebra, and hence by Lemma~\ref{needed} we have  $\tau=n$ for $G(R)$.  By our assumption, $\tau=n$ also for $R$. Thus (\ref{sally}) implies that  $P^R_k(t)=P^{G(R)}_k(t)$. As in (ii) it follows from this equation  that $R$ is Golod, since $G(R)$ is Golod.
 \end{proof}

\end{document}